\documentclass[a4paper,12pt]{article}
\usepackage{amsfonts,epsf,amsmath,amssymb,graphicx}
\usepackage{epstopdf}
\usepackage[top=2.5cm,bottom=2.5cm,left=2.5cm,right=2.5cm]{geometry}
\newtheorem{theorem}{\bf Theorem}[section]

\newtheorem{lemma}[theorem]{\bf Lemma}

\newtheorem{definition}[theorem]{\bf Definition}

\begin{document}
\baselineskip=0.28in

\vspace*{40mm}

\begin{center}
{\LARGE \bf  Resonance graphs of kinky   benzenoid systems are daisy cubes}
\bigskip \bigskip

{\large Petra \v Zigert Pleter\v sek }

\smallskip
{\em  Faculty of Natural Sciences and Mathematics, University of Maribor,  Slovenia} \\
{\em Faculty of Chemistry and Chemical Engineering, University of Maribor, Slovenia} \\

e-mail: {\tt petra.zigert@um.si}

\bigskip\medskip

\today

\end{center}

\vspace{3mm}\noindent

\begin{abstract}
In \cite{kl-mo} introduced daisy cubes  are  interesting isometric  subgraphs of $n$-cubes $Q_n$,  induced with intervals between the maximal elements of a poset $(V(Q_n),\leq)$ and the vertex $0^n \in V(Q_n)$.  
In this paper we show that the resonance graph, which reflects the interaction between Kekul\' e structures of aromatic hydrocarbon molecules, is a daisy cube, if the molecules considered can be modeled with the so called kinky benzenoid systems, i.e. catacondensed benzenoid systems without linear hexagons.
 \end{abstract}
\baselineskip=0.30in

\bigskip

\section{Introduction}

 Klav\v zar and Mollard just recently   introduced in \cite{kl-mo}  a cube like structures called daisy cubes. The daisy cube $ Q_n(X)$ is defined as the subgraph
of $n$-cube $Q_n$ induced by the intersection of the intervals $ I(x, 0^n)$ over all $ x \in X$.
Daisy cubes are partial cubes that include among others Fibonacci cubes, Lucas cubes, and
bipartite wheels.  The authors in \cite{kl-mo} also  introduced  a  bivariate  polynomial called distance cube polynomial and showed   that for a daisy cubes it is connected with the cube polynomial.
	
	Fibonacci and Lucas cubes are a Fibonacci strings based graphs   introduced  as models for 
interconnection networks (see \cite{hsu, hsu-pa,kl} for more details).  Both structures are surprisingly connected with the  chemical molecules made of hydrogen and carbon atoms. Benzenoid hydrocarbons are molecules composed of benzene hexagonal rings and can be viewed as subgraphs of a hexagonal lattice. If we embed them on the  surface of a cylinder
and  roll up a hexagonal lattice we obtain carbon nanotubes - interesting molecules discovered in 1991 \cite{ii}.  Benzenoid hydrocarbons and carbon nanotubes are aromatic molecules and the interaction between theirs  Kekul\'e structures can be modeled with the resonance graph. 
In \cite{kl-zi} the authors showed that the resonance graph of a class of benzenoid hydrocarbons called fibonaccenes is isomorphic to a Fibonacci cube. Later the similar result was shown for a class od carbon nanotubes, but this time it was proven that the largest connected component of the    resonance graph is isomorphic to the Lucas cube, see \cite{zi-be}. Since both classes of cubes, i.e. Fibonaci and Lucas cubes,  are  daisy cubes, we were interested in a problem, for which class of chemical graphs the resonance graph is a daisy cube.

In the next section all the  necessary  definitions are given, followed with the section on the main result of this paper. 
The main result of this paper says, that the resonance graph of a fibonaccene-like benzenoid hydrocarbons, called kinky benzenoid systems, is a daisy cube.

\section{Preliminaries}

Let us first describe the concept of a daisy cube. If  $G = (V (G),E(G))$ is a graph and
$X \subseteq  V (G)$, then $\langle X \rangle$ denotes the subgraph of $ G$ induced by $X$.
Further,  for a word  $u$  of length $n$ over $B=\{0, 1\}$, i.e., $u = (u_1, \ldots, u_n) \in  B^n$
we will briefly write $u$ as $u_1\ldots  u_n$. 
 The $n$-cube $Q_n$ has the vertex set $B^n$, vertices $u_1 \ldots u_n$ and $v_1 \ldots v_n$ being adjacent if $u_i \neq  v_i$ for exactly one $i \in [n]$, where $[n] = \{1, \ldots , n\}$. 
Let $\leq$ be a partial order on $B^n$ defined with $u_1 \ldots u_n\leq  v_1 \ldots v_n $ if $u_i \leq v_i$ holds for all
$i \in  [n]$. For $X \subseteq B^n$ we define the graph $Q_n(X)$ as the subgraph of $Q_n$ with
$Q_n(X) = \left\langle  \{u \in B^n ; u \leq x \,\text{for some} \, x \in X\} \right\rangle $
and say that $Q_n(X)$ is a {\em daisy cube} (generated by $X$).
The vertex sets of daisy cubes are  also  known as hereditary
or downwards closed sets, see \cite{ju}. 
For vertices $u$ and $v$ of a graph $G$ let $d_G(u,v)$ be the  {\em distance} between $u$ and $v$ in $G$.  Further,  the {\em interval} $ I_G(u, v)$ between $u$ and $v$ (in
G) is the set of vertices lying on shortest $u, v$-paths, so  $I_G(u, v) = \{w ; d_G(u, v) =
d_G(u,w)+d_G(w, v)\}$.

{\em Benzenoid systems} are  $2$-connected planar graphs such that every interior face is a hexagon. A { \em benzenoid graph} is the underlying graph of a benzenoid system.   Two hexagons of a benzenoid system  $G$ are said to be {\em adjacent}, if they share an edge. 
A hexagon of $G$  is {\em terminal} if it is adjacent to only one other  hexagon  in $G$.  If  every vertex of a   benzenoid system $G$  belongs to at most two hexagons,   then $G$
is said to be catacondensed; otherwise it is pericondensed.  Note that hexagon $h$ of a catacondensed benzenoid system   that is adjacent to two other
hexagons  contains two vertices of degree 
two. If these two vertices are adjacent, then hexagon $h$ is  called  a {\em  kink} or {\em angularly connected hexagon}, otherwise $h$ is {\em linearly connected hexagon}. A catacondensed benzenoid system with no linearly connected hexagons is called a {\em kinky benzenoid system}.

A {\em 1-factor} of a benzenoid system  $G$ is a
spanning subgraph of $G$ such that every vertex has degree one. The edge set of a 1-factor is called a {\em perfect matching} of $G$, which is a set of independent edges covering all vertices of $G$. In chemical literature, perfect matchings are known as Kekul\'e structures (see \cite{gucy-89} for more details).   A hexagon $h$ of a benzenoid system $G$ is {\em $M$-alternating} if the edges of $h$ appear
alternately in and off the perfect matching $M$. Set $S$ of disjoint hexagons of $G$ is a 
{\em resonant set} of $G$  if there exists a perfect matching $M$  such that  all hexagons in $S$ are $M$-alternating.

Next we will briefly describe the binary coding of perfect matchings of  catacondensed benzenoid system $G$ with $n$ hexagons introduced in \cite{kl-ve}.  First we order the hexagons of a catacondensed benzenoid system with the use of the  depth--first search algorithm (DFS) or by the  breadth--first search algorithm (BFS) (\cite{ko}) performed on the inner dual 
$T$ of $G$, starting at the arbitrary leaf of $T$.  
 Let $h$ and $h'$ be adjacent hexagons of $G$ and $M$ a perfect matching of $G$. Then the 
two edges of $M$  in  $h$ that have exactly one vertex in $h'$ are called the
{\em link} from $h$ to $h'$.  For example, on Figure \ref{slika2} we can see  the link from hexagon $h_3$ to hexagon  $h_2$.  Note that result from \cite{kl-zi-br} says that either both edges of a link belong to some perfect matching of a catacondensed benzenoid system or none.
Further, let ${\cal M}(G)$ be the set of all perfect matchings  of $G$ and we define  a (labeling) function 
$$\ell: {\cal M}(G)\rightarrow \{0,1\}^n$$ 
as follows. Let $M$ be an arbitrary perfect matching of $G$ and let 
$e$ be the edge of hexagon $h_1$ opposite to the common edge of hexagons $h_1$ 
and $h_2$ (see Figure \ref{slika2}). Then for $i=1$ we set 
$$(\ell (M))_1 = \left\{ \begin{array}{ll}
                      1; & e\in M, \\
	              0; & e\notin M\,,
                     \end{array}
             \right.$$
while for $i=2,3,\ldots,n$ we define 
$$(\ell (M))_i = \left\{ \begin{array}{ll}
  1; & M\ {\rm contains\ the\ link\ from}\ h_i\ {\rm to\: its\: \:  predecessor}\ h_{j}, \\
	              0; & {\rm otherwise}\,.
                     \end{array}
             \right.$$
On Figure \ref{slika2}  we can see a kinky benzenoid system with ordered hexagons, a perfect matching $M$ and the  label of $M$. 

\bigskip

\begin{figure}[htb]
	\centering
\includegraphics[scale=0.6]{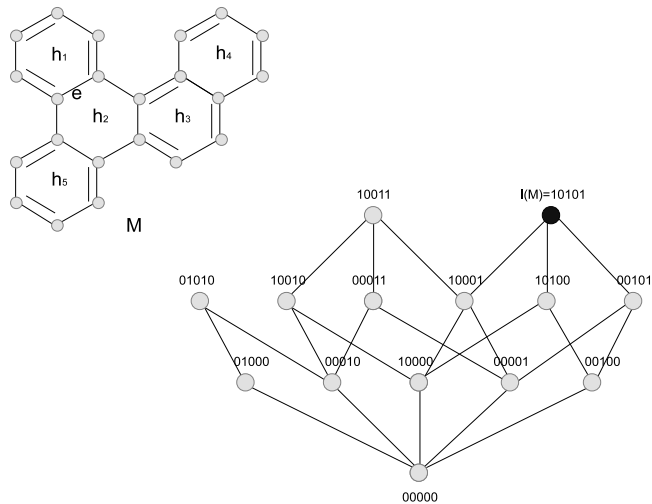}
\caption{A kinky benzenoid system with ordered hexagons, perfect matching $M$ and the  resonance graph with labeled vertices.}
	\label{slika2}
\end{figure}

Let ${\cal L}(G)$ be the set of all binary labels of perfect matchings of a catacondensed benzenoid system $G$. Clearly, ${\cal L}(G)$ is a subset of $B^n=\{0,1\}^n$ and since $(B^n,\leq)$ is a poset,  so  is $({\cal L}(G),\leq)$. 

The vertex set of the {\em resonance graph} $R(G)$ of $G$  consists of all perfect matchings of $G$, two vertices being adjacent whenever their symmetric difference forms the edge set of a hexagon of $G$ (see Figures \ref{slika2}, \ref{slika1}). The concept of the resonance graph was introduced independently in 
mathematics (under the name $Z$-transformation graph)  in \cite{zhgu-88} as well  in  chemistry \cite{elba-93,grun-82,rand-97,rakl-96}.
 We can also introduce the {\em resonance digraph} $\overrightarrow{R}(G)$ of $G$ by orienting the edge $M_1M_2$ of $R(G)$.  An ordered  pair $(M_1,M_2)$ is an arc from $M_1$ to $M_2$  in $\overrightarrow{R}(G)$ if for some $i$ we have 
 $(\ell (M_1))_i=0$ and $ (\ell (M_2))_i=1$,  and the labels of $M_1$ and $M_2$ on all  other positions  (different from $i$-th) coincide. Since the  resonance graph of  a catacondensed benzenoid system with $n$ hexagons  can be isometrically  embedded into the  $n$-cube (in fact it is also  a median graph, see \cite{kl-zi-br}) and  function $\ell$  is the labeling of the vertices of the resonance (di)graph for  that embedding,  it naturally follows that    the Hasse diagram of poset $({\cal L}(G),\leq)$  is isomorphic to the resonance digraph $\overrightarrow{R}(G)$.

\bigskip
\section{The main result}

In this section $G$ is a kinky benzenoid system with $n$ hexagons 
 $\{h_i,h_2,\ldots, h_n\}$ numbered accordingly to the DFS (or BFS) algorithm performed on the inner dual  $T$ of $G$ . Hexagons of $G$ are then numbered so that $h_i$ is a predecessor  of $h_j$ in $T$ if and only if $i < j$.

\begin{lemma}
Let  $G$ be a kinky benzenoid system and  $S$  a maximal resonant set of $G$.  If $h$ is a non-terminal hexagon of $G$ then   either hexagon  $h$ or  one of the hexagons    adjacent to $h$ must be in $S$, and if $h$ is a terminal hexagon of $G$ then either $h$ or the hexagon adjacent to $h$ must be in $S$.

\label{lema1}
\end{lemma}

\begin{proof}
Let $G$ be a kinky benzenoid system and suppose $h$ is a non-terminal hexagon of $G$  where hexagons  $h_{1}, h_{2}$ (and if  exists - $h_{3}$) are adjacent to $h$. Further, let  $S$  be a maximal resonant set of $G$ and $M$ a perfect matching of $G$ such that every hexagon in $S$ is $M$-alternating.  If $h$ is $M$-alternating, we are done. Suppose $h$ is not in $S$ (so $h$ is not $M$-alternating).   Then there must be a link in $M$  from at least one of  hexagons $h_{1}, h_{2}$ (or $h_{3}$)  to hexagon $h$ and consequently this hexagon is $M$-alternating. By the maximality of $S$ it must also belong to $S$.

Similar  holds if $h$ is a  terminal hexagon and $h_{1}$ the hexagon adjacent to it - since $G$ is kinky benzenoid system either $h_{1}$ or $h$  must be a  $M$-alternating hexagon  contained in $S$.
\end{proof}

\bigskip
We can assign in a natural way a binary label to a resonant set of a benzenoid system.
\begin{definition}
Let $G$ be a benzenoid system with  $n$ hexagons $h_i$, $i\in [n]$.
 If $S$  is a resonant set of $G$, then its {\em binary representation } $b(S)$ is a binary string of length $n$ where
 $$(b (S))_j=\left\{ \begin{array}{rcc}  
                1 & ; & h_j \in S,\\
                0 & ; & otherwise\,.
				 \end{array}  \right.$$
\end{definition}


\begin{lemma}
Let $G$ be a kinky benzenoid system and $S$ a resonant set of $G$.  Then $S$  is a maximal resonant set of $G$ if and only if  $b(S)$ is a maximal element in $({\cal L}(G),\leq)$.
\label{lema2}
\end{lemma}

\begin{proof}

First, let $S$ be a maximal resonant set of a kinky benzenoid system  $G$ and $b(S)$ its binary label.  Then we can construct a perfect matching $M$ of $G$ such that   $b(S)=\ell (M)$. 
Namely,  if $(b(S))_i=1$  for some $i \in [n]$, then we insert a link from hexagon $h_i$ to its   adjacent predecessor and $(\ell (M))_{i}=1$ (the same holds if $i=1$ -  then we insert a link from $h_1$ to $h_2$).

On the other side, if $(b(S))_j=0$  then due to Lemma \ref{lema1}
at least one of the hexagons adjacent to hexagon $h_j$ must be in $S$, let it be hexagon $h_{j'}$, so $(b(S))_{j'}=1$.  We again insert a link from hexagon $h_{j'}$ to its adjacent predecessor and  $(\ell (M))_{j'}=1$ (the same holds if $j'=1$ -  then we  insert a link from $h_{j'}=h_1$ to $h_j$). 
  If   $j<j'$, hexagon $h_j$   can not have a link to its   adjacent predecessor and  consequently $(\ell (M))_j=0$ . Also in the  case when  $j'< j$ there is  no link from hexagon $h_j$ to hexagon $h_{j'}$ and again $(\ell (M))_j=0$.

 Suppose  $\ell (M)$ is not a maximal element in 
$({\cal L}(G),\leq)$. Then there exists a perfect matching  $M'$ of $G$ such that
 $\ell (M)\leq \ell (M')$ and $\ell (M) \neq \ell (M')$.  More precisely, there exists $j\in [n]$ so that 
 $(\ell (M'))_j=1$ and $(\ell (M))_j=0$. Hexagon $h_j$ then  does not belong to $S$ and if $h_j$  is  a non terminal hexagon, it is adjacent  to  one   predecessor  $h_{j_0}$   and at most two  successors $h_{j_1}, h_{j_2}$. If $h_j$ is a terminal hexagon it has either only one  adjacent successor  $h_{j_1}$ or only one  adjacent predecessor $h_{j_0}$. 
 
 If $h_j$ is  a non terminal hexagon, then since $(\ell (M'))_j=1$ there is a link from hexagon $h_j$ to hexagon $h_{j_0}$ and since $G$ is kinky benzenoid system  there are no links from  $h_{j_1}, h_{j_2}$ to $h_j$ or from $h_{j_0}$ to its adjacent   predecessor. Consequently   $(\ell (M'))_{j_k}=0, k=0,1,2$.  Also if $h_j$ is  a terminal hexagon adjacent either to one  successor $h_{j_1}$ or one   predecessor $h_{j_0}$,  again $(\ell (M'))_{j_1}=0$ and  $(\ell (M'))_{j_0}=0$, respectively.
 
Let us now consider  label of $M$, $\ell (M)$, and supose $h_j$ is  a non terminal hexagon. 
   Since $h_j$ is not contained in  $S$, then by Lemma \ref{lema1} at least one of the  adjacent hexagons must be in $S$ and is therefore $M$-alternating. So, at least   one of the positions
 $(\ell (M))_{j_k}=1,\: k=0,1,2$ and consequently $\ell (M')$ and $\ell (M)$ are incomparable elements in $({\cal L}(G),\leq)$.
The same argument holds if  $h_j$ is  a terminal hexagon.

For the  if part of a proof suppose $S$ is not a maximal resonant set, so there exists a resonant set $S''$ of $G$ such that $S \subset S''$.  By the same argument as above,  there exist perfect matchings $M$ and $M''$ so that $b(S)=\ell (M)$ and   $b(S'')=\ell (M'')$. 
 But then $\ell (M) \leq \ell (M'')$ and  therefore $\ell (M)=b(S)$ is not a maximal element in 
 $({\cal L}(G),\leq)$.

\end{proof}

\begin{lemma}
Let $G$ be a kinky benzenoid system and  $\ell(M)$  a maximal element in \\
$({\cal L}(G), \leq)$. Then there exists a maximal resonant set $S$ of $G$ such that $b(S)=\ell(M)$.
\label{lema3}
\end{lemma}

\begin{proof}

Let $G$ be a kinky benzenoid system with $n$ hexagons denoted with $h_i$, $i \in [n]$.
Let $\ell (M)$ be a maximal element in $({\cal L}(G),\leq)$, where $M$ is a perfect matching of a $G$. Then  for all  $i \in [n]$ such that  $(\ell (M))_i=1$, the hexagon $h_i$ is $M$-alternating  hexagon of $G$. Let $S$ be the set of all $M$-alternating  hexagons of $G$. 
If $(\ell (M))_j=(\ell (M))_k=1$ for some $j\neq k$ then hexagons $h_j$ and $h_k$ can not be adjacent, since there are links  in $M$ from  hexagons $h_j$ and $h_k$ to theirs adjacent  predecessors. Therefore $S$ is a resonant set. Suppose now $S$ is not a maximal resonant set. Then there exists a hexagon $h'$ of $G$ such that $h \notin S$. By Lemma \ref{lema1} 
$S'=S \cup \{h'\}$ is a resonant set and $b(S) \leq b(S')$.  But then we have perfect matching $M'$ in $G$ such that $b(S')=\ell (M')$ and  $\ell (M) \leq \ell (M')$ contradicts the maximality of $M$.

%
 \end{proof}

\begin{theorem}
The resonance graph  of  a  kinky benzenoid system  is a daisy cube.
\end{theorem}

\begin{proof}

It was shown in  \cite{kl-mo} that the daisy cube $Q_n(X)$ induced on  set $X$ is isomorphic to 
$\langle \cup_{x \in \widehat{X}} I_{Q_n} (x,0^n) \rangle$, where $\hat{X}$ is an antichain consisiting of all maximal elements of $X$.

Let $G$ be a kinky benzenoid system with $n$ hexagons and let $X={\cal L}(G)$ and $\widehat{X}$ the set of all maximal elements of $({\cal L}(G),\leq)$.
By Lemmas \ref{lema2} and \ref{lema3} for each $x=\ell (M) \in \widehat{X}$ there is a maximal resonant set  $S$ of $G$ such that $b(S)=\ell (M)$.
 Further, for  $x =\ell (M)=b(S) \in \widehat{X}$ the graph induced on  interval
$I_{R(G)}(\ell (M),0^n)$ is the subgraph of the  resonance graph $R(G)$ isomorphic to the $k$-cube, where $k$ equals the number of $1'$s in $\ell (M)$ i.e. $k=|S|$ (similar result was shown in \cite{ta-zi}).
 Since for every vertex $\ell (M')$ of the resonance  graph $R(G)$ there exists
  $x' \in \widehat{X}$ such that $\ell (M') \leq x'$, we are done.

\end{proof}

On Figure \ref{slika2} we can see a kinky benzenoid system  with the resonance graph that is a daisy cube, and on Figure \ref{slika1} we have a  benzenoid system with linear hexagon where the  resonance graph is not a daisy cube.

\begin{figure}[!h]
	\centering
	\includegraphics[scale=0.6]{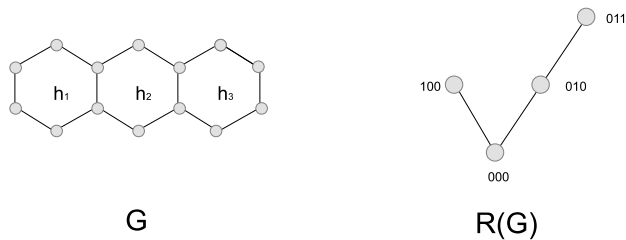}
\caption{A  benzenoid system $G$ with linear hexagon and the resonance graph $R(G)$ which is not a daisy cube.}
	\label{slika1}
\end{figure}

\section*{Acknowledgment}

The author acknowledge the financial support from the Slovenian Research Agency (research core funding No. P1-0297).

\end{document}